\documentclass[11pt,a4paper]{amsart}
\usepackage{amssymb,amsmath,amsthm}
\usepackage{commath}
\usepackage{mathtools}
\usepackage{mathrsfs}
\usepackage{accents}

\usepackage[utf8]{inputenc}

\usepackage[colorlinks]{hyperref}
\hypersetup{
    allcolors = black
}

\newtheorem{theorem}{Theorem}
\newtheorem{corollary}[theorem]{Corollary}

\newtheorem{conjecture}[theorem]{Conjecture}

\theoremstyle{remark}

\newcommand{\RR}{\mathbb{R}}
\newcommand{\R}{\mathbb{R}}
\newcommand{\NN}{\mathbb{N}}

\DeclareMathOperator{\conv}{conv}
\DeclareMathOperator{\diam}{diam}
\DeclareMathOperator{\vol}{vol}
\DeclareMathOperator{\inter}{int}

\author{V\'ictor Hugo Almendra-Hern\'andez, Gergely Ambrus, \\Matthew Kendall}

\title{Quantitative Helly-type theorems via sparse approximation}

\thanks{Research of the second named author was supported by NKFIH grant KKP-133819, by the Ministry of Innovation and
Technology of Hungary from the National Research, Development
and Innovation Fund, project no. TKP2021-NVA-09,  and by the EFOP-3.6.1-16-2016-00008 project, which in turn has been supported by the European Union, co-financed by the European Social Fund.
 }



\date{\today}
\begin{document}

\begin{abstract}
We prove the following sparse approximation result for polytopes. Assume that $Q \subset \R^d$ is a polytope in John's position. Then there exist at most $2d$ vertices of $Q$ whose convex hull $Q'$ satisfies $Q \subseteq - 2d^2 \, Q'$. As a consequence, we retrieve the best bound for the quantitative Helly-type result for the volume, achieved by Brazitikos, and improve on the strongest  bound for the quantitative Helly-type theorem for the diameter, shown by Ivanov and Nasz\'odi: We prove that given a finite family $\mathcal{F}$ of convex bodies in $\R^d$ with intersection $K$, we may select at most $2 d$ members of $\mathcal{F}$  such that their intersection has volume at most $(c d)^{3d /2} \vol K$, and it has diameter at most $2 d^2 \diam K$, for some absolute constant $c>0$.

\end{abstract}

\maketitle

\section{History and results}

Helly's theorem, dated from 1923 \cite{H23}, is a cornerstone result in convex geometry.
Its finitary version states that the intersection of a finite family of convex sets in $\R^d$ is empty if and only if there exists a subfamily of $d+1$ sets such that its intersection is empty.
In 1982, B\'ar\'any, Katchalski and Pach \cite{BKP82} introduced the following quantitative versions of Helly's theorem: there exist positive constants $v(d),\delta(d)$ such that for a finite family $\mathcal{F}$ of convex bodies (that is, compact convex sets with non-empty interior) in $\R^d$, one may select $2d$ members such that their intersection has volume at most $v(d) \operatorname{vol} (\bigcap \mathcal{F})$, or has diameter at most $\delta(d) \operatorname{diam} ( \bigcap \mathcal{F})$.

The problem of finding the optimal values of $\delta(d)$ and $v(d)$ has enjoyed special interest in recent years (see e.g. the excellent survey article \cite{BK21}). In \cite{BKP82} (see also \cite{BKP84}) the authors proved that $v(d) \leq d^{2 d^2}$ and $\delta(d) \leq d^{2d}$, and they conjectured that $v(d) \approx d^{c_1 d}$ and $\delta(d) \approx c_2 d^{1/2}$ for some positive constants $c_1, c_2 >0$.

For the volume problem, in a breakthrough paper, Nasz\'odi  \cite{N16} proved that $v(d) \leq e^{d+1} d^{2d + \frac 1 2}$, while $v(d) \geq d^{d/2}$ must hold. Improving upon his ideas, Brazitikos \cite{B17} found the current best upper bound for volume: $v(d) \le (c d)^{3d/2}$ for a constant $c>0$.

For the diameter question, Brazitikos \cite{B18} proved the first polynomial bound on $\delta(d)$ by showing that $\delta(d) \leq c d^{11/2}$ for some $c>0$. In 2021, Ivanov and Nasz\'odi \cite{IN21} found the best known upper bound, $\delta(d) \le (2d)^{3}$, and also proved that $\delta(d) \geq c d^{1/2}$. Thus, the value conjectured in  \cite{BKP82} for $\delta(d)$ would be asymptotically sharp.

In the present note, we show that given a finite family $\mathcal{F}$ of closed convex sets, one can select at most $2d$ members such that their intersection sits inside a scaled version of $\bigcap \mathcal{F}$ for a suitable location of the origin.
Clearly, it suffices to prove this statement for the special case when $\mathcal{F}$ consists of closed halfspaces intersecting in a convex body.
As an application, we obtain an improvement on the diameter bound, $\delta(d) \le 2d^2$, and retrieve the best known bound for $v(d)$.
The crux of the argument is the following  sparse approximation result for polytopes, which is a strengthening of Theorem 2 in \cite{IN21}.

\begin{theorem}
    \label{thm:homothet}
    Let $ \lambda > 0$ and $Q \subset \RR^d$ be a convex polytope
    such that $Q \subseteq - \lambda Q$.
    Then there exist at most $2d$ vertices of $Q$ whose convex hull $Q'$ satisfies
        \begin{equation*}
        Q \subseteq - ( \lambda + 2)d \, Q'.
    \end{equation*}
\end{theorem}

We immediately obtain the following corollary.

\begin{corollary}\label{cor1}
Assume that $ Q = -Q$ is a symmetric convex polytope in $\RR^d$. Then we may select a set of at most $2d$ vertices of $Q$ with convex hull $Q'$ such that
\begin{equation*}
        Q \subseteq 3 d \, Q'.
    \end{equation*}
\end{corollary}

As usual, let $B^d$ denote the unit ball of $\RR^d$ and let $S^{d-1}$ be the unit sphere of $\RR^d$. A standard consequence of Fritz John's theorem \cite{J48} states that if $K \subset \RR^d$ is a convex body in John's position, that is, the largest volume ellipsoid inscribed in $K$ is $B^d$,  then $B^d \subseteq K \subseteq d B^d \subseteq - d K$ (see e.g. \cite{B97}). Thus, we derive

\begin{corollary}\label{cor2}
Assume that $ Q \subset \RR^d$ is a convex polytope in John's position. Then there exists a subset of at most $2d$ vertices of $Q$ whose convex hull $Q'$ satisfies
\begin{equation*}
        Q \subseteq - 2 d^2 \, Q'.
    \end{equation*}
\end{corollary}

\medskip

For $n \in \NN^+$, we will use the notation $[n] = \{1, \dots,n\}$.
For a family of sets $\{ K_1, \ldots, K_n\} \subset \RR^d$ and for a subset $\sigma \subset [n]$, let
\[
K_\sigma = \bigcap_{i \in \sigma} K_i,
\]
as in \cite{IN21}. Also, let $\binom{[n]}{\leq k}$ stand for the set of all subsets of $[n]$ with cardinality at most $k$. Using this terminology, we are ready to state our quantitative Helly-type result.

\begin{theorem}
\label{thm:main}
    Let $\{K_1,\ldots,K_n\}$ be a family of closed convex sets in $\RR^d$ with $d \geq 2$  such that their intersection $K = K_1 \cap \cdots \cap K_n$ is a convex body.
    Then there exists a $\sigma \in \binom{[n]}{\le 2d}$ such that
    \begin{equation*}
        \vol_d K_{\sigma} \le (cd)^{3d/2} \vol_d K \;\;\text{and}\;\;
        \diam K_{\sigma} \le 2 d^2 \diam K
    \end{equation*}
    for some constant $c > 0$.
\end{theorem}

To conclude the section we formulate the following conjecture, which was essentially stated already in \cite{BKP82}. This would imply the asymptotically sharp bound for $v(d)$, see the Remark after the proof of Theorem~\ref{thm:main}.

\begin{conjecture}\label{conj:john-pos}
Assume that $\{u_1, \ldots, u_n \}\subset S^{d-1}$ is a set of unit vectors satisfying the conditions of Fritz John's theorem. That is, there exist positive numbers $\alpha_1, \ldots, \alpha_n$ for which $\sum_{i=1}^n \alpha_i u_i =0$ and $\sum_{i=1}^n \alpha_i u_i \otimes u_i = I_d$, the identity operator on $\RR^d$. Then there exists a subset $\sigma \subset [n]$ with cardinality at most $2d$ so that
\[
 B^d \subset c \, d \conv \{ u_i: i \in \sigma  \}
\]
with an absolute constant $c>0$.
\end{conjecture}

That the above estimate would be asymptotically sharp is shown by taking $n=d+1$ and letting $\{u_1, \ldots, u_n \}$ to be the set of vertices of a regular simplex inscribed in $S^{d-1}$.

Note that we study quantitative Helly-type questions that require selecting {\em at most $2d$} sets, which is the smallest cardinality for which such estimates may hold. Versions obtained by relaxing this cardinality bound have been studied e.g. by Brazitikos~\cite{B17b}, Dillon and Soberón~\cite{DS21} and Ivanov and Naszódi~\cite{IN21}. In particular, an estimate which matches Theorem~\ref{thm:homothet} asymptotically was given in~\cite{IN21} when selecting $2d+1$ vertices of the polytope, and  an asymptotically sharp estimate for the quantitative
Helly-type theorem for the diameter was proved in~\cite{DS21} for sufficiently large sub-families. Further quantitative Helly-type results have been studied in~\cite{IN22b} (for log-concave functions) and \cite{VGM22} (continuous versions).

\section{Proofs}

\begin{proof}[Proof of Theorem~\ref{thm:homothet}]
    The condition $Q \subseteq -\lambda Q$ ensures that $0 \in \inter Q$.
    Among all simplices with $d$ vertices from the set of vertices of $Q$ and one vertex at the origin, consider a simplex $S = \conv\{0,v_1,\ldots,v_d\}$ with maximal volume.
    We write $S$ in the form
    \begin{equation}\label{eq:conv_S}
        S = \bigg\{ x \in \RR^d : x = \alpha_1 v_1 + \ldots + \alpha_d v_d \textrm{ for } \alpha_i \geq 0 \textrm{ and } \sum_{i=1}^d \alpha_i \leq 1 \bigg\}.
    \end{equation}
    For every $i \in [d]$, let $H_i$ be the hyperplane spanned by $\{0, v_1,\ldots,v_d\} \setminus \{v_i\}$,
    and let $L_i$ be the strip between the hyperplanes $v_i + H_i$ and $ -v_i + H_i$.
    Define $P = \bigcap_{i \in [d]} L_i$ (see Figure~\ref{fig1}).

    Note that
    \begin{equation}
        \label{eq:parallelotope_1}
        P = \{ x \in \RR^d : \vol_d(\conv( \{0,x,v_1,\ldots,v_d\} \setminus \{v_i\} ) \le \vol_d(S) \; \textrm{for all $i \in [d]$}\}.
    \end{equation}
    This follows from the volume formula
     \[
        \vol_d(\conv\{0,w_1,\ldots,w_d\}) =
        \frac{1}{d!} \big| \hspace{-1 pt} \det (w_1 \, w_2 \, \cdots \, w_d)\big|
    \]
    for arbitrary $w_1,\ldots,w_d \in \RR^d$, which implies that for all $x \in \RR^{d}$ of the form  $ x = c v_i + w$ with $w \in H_i$, $i \in [d]$,
    \[
        \vol_d(\conv( \{0,x,v_1,\ldots,v_d\} \setminus \{v_i\} ) = |c| \vol_d(S).
    \]

    Next, we show that
    \begin{equation}\label{eq:conv_P}
        P = \{ x \in \RR^d : x = \beta_1 v_1 + \ldots + \beta_d v_d \textrm{ for } \beta_i \in [-1,1] \}.
    \end{equation}
    Indeed, since $v_1,\ldots,v_d$ are linearly independent, we may consider the linear transformation $A$ with $A(v_i) = e_i$ for  $ i \in [d]$.
    Note that
    \begin{align*}
        A(P) = A \bigg( \bigcap_{i \in [d]} L_i \bigg)
        = \bigcap_{i \in [d]} A(L_i)
        = \{ x \in \RR^d : x = \beta_1 e_1 + \cdots + \beta_d e_d  \textrm{ for } \beta_i \in [-1,1]\}.
    \end{align*}
    Thus, \eqref{eq:conv_P} holds.

\begin{figure}[h]
\label{figure1}
\centering
  \includegraphics[width=0.55\textwidth]{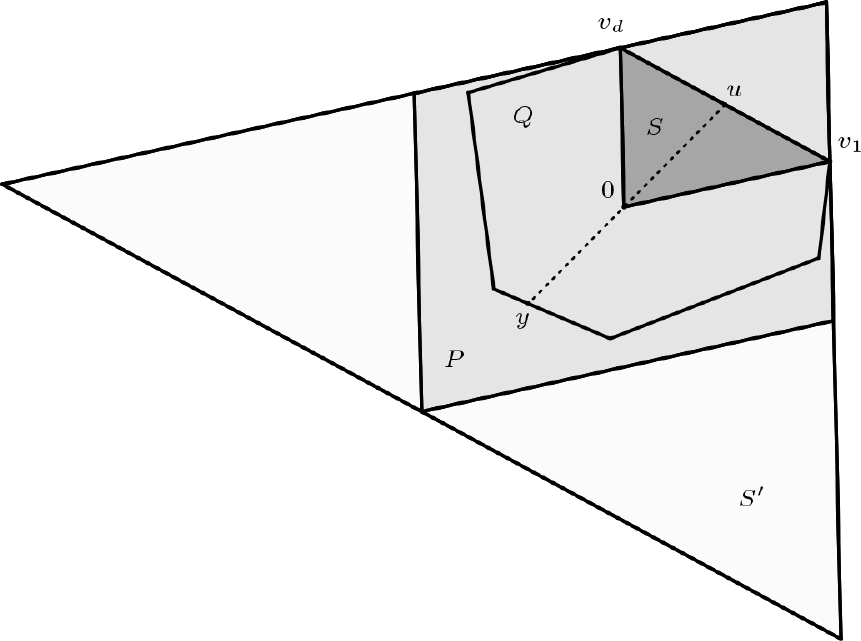}
\caption{}
\end{figure}

    Since $S$ is chosen maximally, equation \eqref{eq:parallelotope_1} shows that for any vertex $w$ of $Q$, $w \in P$. By convexity,
    \begin{equation}
        \label{eq:1}
        Q \subseteq P.
    \end{equation}
    Let $S' = -2dS + (v_1 + \ldots + v_d)$. By \eqref{eq:conv_S},
    \begin{equation}\label{eq:conv_S'}
        S' = \bigg\{x \in \RR^d : x = \gamma_1v_1 + \ldots + \gamma_dv_d \text{ for } \gamma_i \leq 1 \textrm{ and } \sum_{i= 1}^{d} \gamma_i \geq -d \bigg\}.
    \end{equation}
    Then, from \eqref{eq:conv_P} and \eqref{eq:conv_S'},
    \begin{equation}
        \label{eq:2}
        P \subseteq S'.
    \end{equation}

    Let $u = \frac{1}{d}(v_1 + \ldots + v_d)$ be the centroid of the facet $\conv\{v_1, \dots, v_d\}$ of $S$.
    Let $y$ be the intersection of the ray from $0$ through $-u$
    and the boundary of $Q$.
    By Carath\'eodory's theorem, we can choose $k \le d$ vertices $\{v_1',\ldots,v_k'\}$ of $Q$ such that
    $y \in \conv\{v_1',\ldots,v_k'\}$.
    Set $Q'= \conv \{v_1,\ldots,v_d,v_1',\ldots,v_k'\}$.


    Note that $[y,u] \subseteq Q'$, which implies $0 \in Q'$.
    Thus,
    \begin{equation}
        \label{eq:3}
        S \subseteq Q'.
    \end{equation}
    Since $Q \subseteq - \lambda Q$, we have that $-u \in  \lambda Q$.
    Since $ \lambda y$ is on the boundary of $ \lambda Q$, we also have that
    $-u \in [0, \lambda y]$.
    We know that $0, \lambda y \in \lambda  Q'$, so
    \begin{equation}
        \label{eq:5}
        u \in - \lambda  Q'.
    \end{equation}
    Combining \eqref{eq:1}, \eqref{eq:2}, \eqref{eq:3}, and
    \eqref{eq:5}:
    \begin{equation} \label{eq:6}
        Q \subseteq P \subseteq S'
        = -2d S + d  u
        \subseteq -2d  \, Q' - \lambda d \, Q'
        = -(\lambda + 2) d \, Q'.
        \qedhere
    \end{equation}
\end{proof}

\begin{proof}[Proof of Theorem~\ref{thm:main}]
    As shown in \cite{BKP82}, we may assume that the family  $\{K_1,\ldots,K_n\}$ consists of closed halfspaces such that $K = \bigcap K_i$ is a $d$-dimensional polytope.
    Let $T$ be the affine transformation sending $K$ to John's position.
    Let $\widetilde{K}_i = T{K}_i$ for $i \in [n]$, $\widetilde{K} = TK$, and for some $\sigma \subset [n]$, let $\widetilde{K}_{\sigma} = \bigcap_{i \in \sigma} \widetilde{K}_i$.
    We claim that there exists $\sigma \in \binom{[n]}{\le 2d}$ such that the following two properties hold:
    \begin{align}
        \widetilde{K}_{\sigma} &\subseteq -2d^2 \widetilde{K}, \textrm{ and} \label{eq:inclusion} \\
        \vol_d \widetilde{K}_{\sigma} &\le (cd)^{3d/2} \vol_d \widetilde{K} \label{eq:volume}
    \end{align}
    for some constant $c > 0$.
    Estimates~\eqref{eq:inclusion} and~\eqref{eq:volume} are affine invariant, so this will suffice to prove Theorem~\ref{thm:main}.

    Recall that since $\widetilde{K}$ is in John's position, $B^d \subseteq \widetilde{K} \subseteq d B^d$ (see \cite{B97} or \cite[Theorem 5.1]{GLMP04}).
    Setting $Q = (\widetilde{K})^{\circ}$, this yields that $\frac 1 d B^d \subseteq Q \subseteq B^d$ (here and later on, $K^\circ$ stands for the polar set: $K^\circ = \{ x \in \R^d: \langle x, y \rangle \leq 1 \ \forall y \in K \}$.)
    In particular, $Q \subseteq -d Q$. Hence, we may apply Theorem~\ref{thm:homothet} to $Q$ with $\lambda = d$, we obtain a subset of at most $2d$ vertices of $Q$ such that their convex hull $Q'$ satisfies
    \begin{equation}
    \label{eq:inclusion_1}
        Q \subseteq -(d + 2)d  Q' \subseteq - 2 d^2 Q'.
    \end{equation}
    The family of closed halfspaces supporting the facets of $(Q')^{\circ}$ is a subset of $\{\widetilde{K}_1,\ldots, \widetilde{K}_n\}$ with at most $2d$ elements.
    Thus, we can choose $\sigma \in \binom{[n]}{\le 2d}$ such that $\widetilde{K}_{\sigma} = (Q')^{\circ}$.
        Taking the polar of~\eqref{eq:inclusion_1}, we obtain
    \begin{equation*}
        \widetilde{K}_{\sigma} \subseteq - (d+2)d \widetilde{K} \subseteq -2d^2 \widetilde{K},
    \end{equation*}
    which shows~\eqref{eq:inclusion}.

    Let $P$ be the parallelotope enclosing $Q$ from the proof of Theorem~\ref{thm:homothet} and set $P' = -\frac{1}{2d^2} P$.
    Statement \eqref{eq:6} implies
    \[
        Q' \supseteq P'.
    \]
    Since $S$ is chosen maximally, the volume of $S$ is at least the volume of the simplex obtained from the Dvoretzky--Rogers lemma \cite{DR50} (see also \cite[Lemma 1.4]{N16}):
    \begin{equation}\label{eq:DR_lemma}
        \vol_d(S) \ge \frac{1}{\sqrt{d!}d^{d/2}}.
    \end{equation}
    Using \eqref{eq:DR_lemma},
    \begin{equation}\label{eq:vol_P'}
        \vol_d(P')
        = (2d^2)^{-d} \vol_d(P)
        = (2d^2)^{-d} \cdot 2^{d} d! \vol_d(S)
        \ge d^{-5d/2} (d!)^{1/2}.
    \end{equation}
    Note that $P'$ is centrally symmetric, so we can use the Blaschke--Santal\'o inequality (see \cite[Theorem 1.5.10]{AGM15}) for $P'$:
    \begin{equation}
        \label{eq:santalo}
        \vol_d (P') \cdot \vol_d ((P')^{\circ}) \le \vol_d (B_2^d)^2.
    \end{equation}
    Using the inclusions $\widetilde{K} \supseteq B_2^d$ and
    $\widetilde{K}_{\sigma} = (Q')^{\circ} \subseteq (P')^{\circ}$,
    as well as \eqref{eq:vol_P'} and  \eqref{eq:santalo}:
    \begin{equation}\label{eq:final}
        \frac{\vol_d \widetilde{K}_{\sigma}}{\vol_d \widetilde{K}}
        \le \frac{ \vol_d( (P')^{\circ})}{\vol_d( B_2^d)}
        \le \frac{ \vol_d( B_2^d)}{ \vol_d(P')}
        \le \frac{\pi^{d/2} d^{5d/2} (d!)^{-1/2} }{\Gamma((d/2) + 1)}
        \le (cd)^{3d/2}
    \end{equation}
    for some absolute constant $c > 0$.
    This shows \eqref{eq:volume} and concludes the proof.
\end{proof}

\noindent
{\em Remark.}
We briefly explain how Conjecture~\ref{conj:john-pos} would imply the asymptotically optimal bound on $v(d)$.  First note that the estimate \eqref{eq:inclusion_1} would hold with  the factor $cd$ instead of $2d^2$. Then, in the rest of the proof of Theorem~\ref{thm:main}, we could replace all instances of the  factor $2 d^2$  with $ cd$. In particular, one would get the linear upper bound  $\delta(d) \leq c d$  from the improvement of \eqref{eq:inclusion}, while the rest of the calculations would show that the final quotient in equation~\eqref{eq:final} is  at most $(c' d)^{d/2}$ for some absolute constant $c' > 0$.

\section{Acknowledgements}

This research was done under the auspices of the
Budapest Semesters in Mathematics program. We are grateful to the anonymous referees for their valuable comments on the article.

\bigskip

\bigskip

\bigskip

\noindent
{\sc V\'ictor Hugo Almendra-Hern\'andez}
\smallskip

\noindent
{\em Facultad de Ciencias, Universidad Nacional Aut\'onoma de M\'exico, Ciudad de M\'exico, M\'exico}
\smallskip

\noindent
e-mail address: \texttt{vh.almendra.h@ciencias.unam.mx}

\bigskip

\noindent
{\sc Gergely Ambrus}
\smallskip

\noindent
{\em Alfréd Rényi Institute of Mathematics, Eötvös Loránd Research Network, Budapest, Hungary and\\ Bolyai Institute, University of Szeged, Hungary}
\smallskip

\noindent
e-mail address: \texttt{ambrus@renyi.hu}

\bigskip

\noindent
{\sc Matthew Kendall}
\smallskip

\noindent
{\em Department of Mathematics, Princeton University, Princeton, NJ, USA}
\smallskip

\noindent
e-mail address: \texttt{mskendall@princeton.edu}

\end{document}